\documentclass[12pt]{article}

\usepackage{geometry}                
\usepackage{amsmath,amsfonts,amssymb,amsthm,amscd,enumerate}
\usepackage[all]{xy}
\usepackage{graphicx}
\usepackage{epstopdf}

\usepackage[dvipsnames]{color}

\usepackage{ulem}

\title{K-theory and index pairings for C*-algebras generated by q-normal operators}

\author{{\sc Ismael Cohen} \\
\normalsize
Instituto de F\'isica y Matem\'aticas\\
\normalsize
Universidad Michoacana de San Nicol\'as de Hidalgo, Morelia, M\'exico\\[2pt]
\normalsize
and\\[2pt]
\normalsize
Centro de Ciencias Matem\'aticas, Campus Morelia\\
\normalsize
Universidad Nacional Aut\'onoma de M\'exico (UNAM), Morelia, M\'exico\\
\normalsize
e-mail: {\it ismaelcohen10@gmail.com}\\[16pt] 
{\sc Elmar Wagner\footnote{
corresponding author \ \  {\it MSC2010:} 46L85, 46L80, 46L65, 58B32 \ \  
{\it Key Words:} q-normal op\-era\-tors, quantum complex plane, 
quantum line bundles, K-theory, K-homology, index pairings}
} \\
\normalsize
Instituto de F\'isica y Matem\'aticas\\
\normalsize
Universidad Michoacana de San Nicol\'as de Hidalgo, Morelia, M\'exico\\
\normalsize
e-mail: {\it elmar@ifm.umich.mx}}

\date{}                                           

\newtheorem{thm}{Theorem}
\newtheorem{prop}[thm]{Proposition}

\newtheorem{cor}[thm]{Corollary}

\theoremstyle{definition}

\newtheorem{rem}[thm]{Remark}

\newcommand{\nc}[2]{\newcommand{#1}{#2}}
\newcommand{\rnc}[2]{\renewcommand{#1}{#2}}
\rnc{\[}{\begin{equation}}
\rnc{\]}{\end{equation}}
\nc{\wegengruen}{\end{equation}}

\newcommand{\Z}{\mathbb{Z}}
\newcommand{\N}{\mathbb{N}}
\newcommand{\R}{\mathbb{R}}
\newcommand{\C}{\mathbb{C}}
\newcommand{\hsp}{{\hspace{-1pt}}}
\newcommand{\hs}{{\hspace{1pt}}}

\newcommand{\hH}{\mathcal{H}}
\newcommand{\cO}{\mathcal{O}}

\newcommand{\K}{{\mathcal{K}}}

\newcommand{\A}{{\mathcal{A}}}

\newcommand{\Cq}{\cO(\C_q)}

\newcommand{\im}{\mathrm{i}}
\newcommand{\e}{\mathrm{e}}
\newcommand{\rmc}{\mathrm{c}}
\newcommand{\ev}{\mathrm{ev}}
\newcommand{\lin}{{\mathrm{span}}}
\newcommand{\id}{{\mathrm{id}}}
\newcommand{\ind}{{\mathrm{ind}}}

\newcommand{\spec}{\mathrm{spec}}
\newcommand{\dom}{\mathrm{dom}}

\newcommand{\Tr}{{\mathrm{Tr}}}

\newcommand{\rmB}{\mathrm{B}}
\newcommand{\rmK}{\mathrm{K}}
\newcommand{\Mat}{\mathrm{Mat}}

\newcommand{\lZ}{\ell_2(\Z)}

\newcommand{\eins}{\hs {\bf 1}}

\newcommand{\ra}{\rightarrow}
\newcommand{\lra}{\longrightarrow}

\newcommand{\ip}[2]{\langle{#1},{#2}\rangle}
\newcommand{\msum}[2]{\underset{{#1}}{\overset{{#2}}{\mbox{$\sum$}}}}

\begin{document}
\maketitle




\begin{abstract}
The paper presents a detailed description of the K-theory and K-homology of 
C*-algebras generated by $q$-normal operators including generators 
and the index pairing. The C*-algebras generated by $q$-normal operators can be viewed 
as a $q$-deformation of the quantum complex plane. In this sense,  we find 
deformations of the classical Bott projections describing complex 
line bundles over the 2-sphere, but there are also simpler generators 
for the $K_0$-groups, for instance 1-dimensional Powers--Rieffel type projections and 
elementary projections belonging to the C*-algebra. 
The index pairing between these projections and generators of the 
even K-homology group is computed, and the result is used to 
express the $K_0$-classes of 
the quantized line bundles of any winding number 
in terms of the other projections. 
\end{abstract}

\section{Introduction}
In this paper we give a complete description of the K-theory of 
all possible C*-al\-geb\-ras generated 
by one of the most prominent relations 
occurring in the theory of $q$-deformed spaces: 
\[                                                                                                                             \label{z}
z\hs z^* =q^2 \hs z^* \hs z, \qquad q \in(0,1). 
\]

The complex *-algebra with generators $z$ and $z^*$ subject to the relation \eqref{z} 
is known as the coordinate ring  $\Cq$ of the quantum complex plane. 
On the analytic side, a densely defined closed linear operator on a Hilbert space satisfying~\eqref{z} 
is called  $q$-normal operator. In other words, any $q$-normal operator yields a 
Hilbert space representation of $\Cq$. 
These representations  have been classified in 
\cite{CSS} and \cite{CW}. The results therein include that non-zero $q$-normal operators are never bounded.
As a consequence, $\Cq$ cannot be equipped with a C*-norm, 
which is consistent with the idea that $\Cq$ should be viewed as the coordinate ring of a 
{\it non}-compact quantum space. 
To study non-compact quantum spaces in the C*-algebra setting, 
one may apply Woronowicz's theory 
of C*-algebras generated by unbounded elements \cite{Wo}. 
This has been done by the authors in \cite{CW}. It turned out that 
the C*-algebra generated by a $q$-normal operator $z$ depends only on the spectrum of 
the self-adjoint operator $|z|= \sqrt{z^*z}$. Among all these C*-algebras, there is a universal one, 
namely when the spectrum of $|z|$ coincides with the whole interval $[0,\infty)$. 
This algebra is viewed as the algebra of continuous functions vanishing 
at infinity on the quantum complex plane and is denoted by $C_0(\C_q)$. 

A natural question is if the passage from the commutative C*-algebra $C_0(\C)$ 
to the $q$-deformed version  $C_0(\C_q)$ preserves topological invariants. 
Preserving topological invariants gives another justification for calling 
$C_0(\C_q)$ the algebra of continuous functions vanishing 
at infinity on the quantum complex plane. 
Nevertheless, we are also interested in detecting quantum effects, i.e., 
situations where the computations of invariants differ from the 
classical case. In favorable situations, the quantum case even leads 
to a simplification. 

To provide answers to these questions, we compute the K-theory for all  
C*-al\-ge\-bras generated by $q$-normal operators. 
Our results in Theorem \ref{T1} show that $C_0(\C_q)$ does actually have the same K-groups 
as the commutative C*-algebra $C_0(\C)$. 
The analogy goes even further since the non-zero elements of 
$K_0(C_0(\C_q))$  can be described by $q$-deformed versions  
of the classical Bott projections describing complex 
line bundles of winding number $n\in\Z$ over the 2-sphere.  
However, as a quantum effect, the $K_0$-classes can also be given by 1-dimensional 
Power--Rieffel type projections in the C*-algebra $C_0(\C_q)$. 
Note that $C_0(\C)$ does not contain non-trivial projections since $\C$ is connected. 

The situation changes if $\spec(|z|) \neq [0,\infty)$. 
Then the $K_1$-group of the C*-algebra 
generated  by the $q$-normal operator $z$ is trivial 
as in the classical case 
but the $K_0$-group 
depends on the number of gaps in the set $\spec(|z|) \cap [q,1]$,  
and any of the groups $\Z^n$, $n>1$, as well as $\oplus_{n\in\N} \Z$  can occur (Theorem \ref{KCzz}). 
Moreover, the $K_0$-groups are generated by elementary 1-dimensional projections 
belonging to the C*-algebra. If one wants to consider all C*-algebras generated 
by a $q$-normal operator as deformations of the complex plane, then  
quantization leads to a whole family of quantum spaces 
with different topological properties 
and also simplifies the description of generators of the $K_0$-groups. 

A practical method of determining $K_0$-classes is by computing the index pairing with 
K-homology classes. We present for all C*-algebras generated by a $q$-normal operator 
a set of generators for the (even) K-homology group 
and show that it gives rise to a non-degenerate pairing with the $K_0$-group. 
The index pairing between these generators and all the projections mentioned above is computed 
(Theorem \ref{TIP} and Theorem \ref{Teo}), 
and the result is used to express the $K_0$-class of the quantized line bundles of winding number $n\in\Z$ 
in terms of the different projections (Corollary~\ref{corPRchi}). 
Moreover, one can find generators of the even K-homology group that compute exactly 
the rank or the winding number of the quantized line bundles. 
Remarkably, for the elementary projections, the computation of the winding number 
 boils down to its simplest form: the computation of a trace 
of a projection onto a finite dimensional subspace. 
Thus one can say that quantization leads to a significant simplification of the index computation. 

The description of the K-groups is only the first step, albeit an essential one, of the larger program 
of understanding the noncommutative geometry of the quantum complex plane. 
A further step would be to find a Dirac operator satisfying the axioms of a spectral triple 
which might be a noncommutative analogue of the Dirac operator on~$\R^2$ with 
the flat metric or of the Dirac operator on the Riemannian 2-sphere in local coordinates. 
In view of a possible $q$-deformed differential calculus associated to 
the commutation relation \eqref{z}, 
it seems to be natural to look for a twisted spectral triple 
in the sense of Connes and Moscovici \cite{CM}. 
However, this admittedly more difficult problem is beyond the scope of the present paper.

\section{C*-algebras generated by q-normal operators} 

\label{sec-qn} 

In this section we collect the most important facts on C*-algebra generated by q-normal operators from \cite{CW} 
(see also \cite{CSS}). 
Let $q\in (0,1)$ and $z$ be a $q$-normal operator, that is, $z$ is a densely defined closed linear operator  
on a Hilbert space $\hH$ such that \eqref{z} holds on $\dom(z^*z)= \dom(zz^*)$. 
By \cite[Corollary 2.2]{CW}, 
the Hilbert space $\hH$ decomposes into the direct sum 
$
\hH = \ker(z)\, \oplus \underset{n\in\Z}{\oplus} \hH_n, 
$
where, up to unitary equivalence, we may assume $\hH_n=\hH_0$. 
For $h\in \hH_0$ and $n\in\Z$ , let $h_n$ denote the vector in $\underset{n\in\Z}{\oplus} \hH_n$ which has 
$h$ in the $n$-th component and $0$ elsewhere. Then the action of $z$ on  $\hH$   
is determined by 
\[ \label{zh} 
z=0 \ \text{ on }\   \ker(z),\qquad    z\hs h_n = q^n (Ah)_{n-1} \ \ \text{on}\ \ \underset{n\in\Z}{\oplus} \hH_n, 
\] 
where $A$ is a self-adjoint operator $A$ on $\hH_0$ such that  $\spec(A)\subset [q,1]$ and  $q$ is not an eigenvalue. 
A non-zero representation of a $q$-normal operator $z$ is irreducible if and only if $\ker(z) =\{0\}$ and 
$\hH_0=\C$. In this case, $A$ can be viewed as a real number in $(q,1]$. 

In \cite[Section 3]{CW}, it has been shown that the C*-algebra generated in the sense of Woro\-no\-wicz \cite{Wo} by a 
$q$-normal operator $z$  depends only on the spectrum 
\[ \label{X} 
X:= \spec(|z|) = \{0\}\cup \underset{n\in\Z}{\cup} q^n \hs \spec(A)    \subset [0,\infty).
\]  
More than that, it can be described as  a C*-subalgebra of the crossed product algebra $C_0( X)\rtimes \Z$ 
without referring to the Hilbert space $\hH$.  Here, for any  $q$-invariant locally compact subset $ X\subset [0,\infty]$ 
(such as $\spec(|z|)$),  the $\Z$-action is given by the automorphism 
\[   \label{alpha}
\alpha_q: C_0(X) \lra C_0(X), \qquad \alpha_q(f)(x):= f(qx). 
\] 

Recall that the crossed product algebra $C_0( X) \rtimes  \Z$ 
can be described as 
enveloping C*-algebra generated by functions $f\in C_0(X)$ and a unitary operator $U$ subjected to the relation 
$$ \label{UfU} 
U^* fU = \alpha_q(f). 
$$
By \cite[Theorem 3.2]{CW}, the C*-algebra generated by a non-zero 
$q$-normal operator $z$ is isomorphic to 
\[ \label{Czz}
C^*_0(z,z^*):=\|\cdot\|\text{-}\mathrm{cls}\hs \Big\{ 
\sum_{\text{{\rm finite}}} f_k \hs U^k\in C_0( X)\rtimes \Z \;:\; k\in\Z,\,\  f_k(0)=0\ \, 
\text{if}\ \,k\neq 0 
\Big\}, 
\]
where $X= \spec(|z|)$ and  
$\|\cdot\|\text{-}\mathrm{cls}$ denotes the norm closure in 
$C_0( X)\rtimes \Z$.   
The case $\spec(|z|)= [0,\infty)$ has the universal property that 
$$
C_0([0,\infty))\rtimes \Z \ni \sum_{\text{{\rm finite}}} f_k \hs U^k \;
\longmapsto\; \sum_{\text{{\rm finite}}} f_k \!\!\upharpoonright_{X} \! U^k  \in C_0( X)\rtimes \Z
$$
yields a well-defined *-homomorphism of crossed product algebras which restricts to the corresponding 
C*-subalgebras defined in \eqref{Czz}.   This was one of the motivations in \cite{CW} 
to define the C*-algebra algebra of 
continuous functions vanishing at infinity on a quantum complex plane as 
the C*-algebra generated by a $q$-normal operator $z$ satisfying $\spec(|z|)=[0,\infty)$, i.e., 
\[  \label{C0Cq}
C_0(\C_q):=\|\cdot\|\text{-}\mathrm{cls}\Big\{ 
\sum_{\text{{\rm finite}}} f_k \hs U^k\in C_0( [0,\infty))\rtimes \Z \,:\, f_k(0)=0\ \, 
\text{if}\ \,k\neq 0 
\Big\} .  
\] 
Furthermore, its unitization  
\[  \label{CSq} 
C(\mathrm{S}^2_q):= C_0(\C_q)\dotplus \C\eins
\]
is viewed as the C*-al\-ge\-bra of continuous functions on  a quantum 2-sphere  
obtained from a one-point compactification of the 
quantum complex plane. 
To distinguish  $C_0(\C_q)$ from the C*-algebras $C^*_0(z,z^*)$ generated by 
$q$-normal operator $z$ such that  $\spec(|z|)\neq [0,\infty)$, we call 
the latter case {\it generic}. 

By \eqref{zh},  $0\in X:=\spec(|z|)$  for any $q$-normal operator $z$. 
As  0 is in\-var\-iant under multiplication by $q$, 
\[     \label{ev0}
\ev_0: \Big\{\sum_{\text{{\rm finite}}} f_k \hs U^k : f_k\in C_0(X) \Big\} \lra \C, \quad 
\ev_0\Big(\sum_{\text{{\rm finite}}} f_k \hs U^k\Big) \,=\, \sum_{\text{{\rm finite}}} f_k (0)
\]
yields a well-defined *\hspace{-0.25pt}-homomorphism, 
in particular a so-called covariant representation, where one may set $\ev_0(U):=1$. Since the C*-norm of the enveloping C*-algebra 
is given by taking the supremum of the operator norms over all *-representations \cite{Will}, 
the map \eqref{ev0} is norm decreasing and thus 
extends to a continuous *-homomorphism $\ev_0: C_0( X) \rtimes \Z \ra \C$. 
From the definitions of 
$C_0( X) \rtimes \Z$, \,$C^*_0(z,z^*)$ and  $\ev_0$, we get the following 
exact sequence of C*-algebras: 
\begin{equation}\label{ext}
\xymatrix{
 0\;\ar[r]&\; C_0(X \setminus \{0\})\rtimes \Z\;\ar@{^{(}->}[r]^{\quad\mathrm{\iota}} &\; C^*_0(z,z^*) 
\;\ar[r]^{\ \ \ \ev_0} &\; \C\;\ar[r] &\,0\,.}
\end{equation}
For $C^*_0(z,z^*)\cong C_0(\C_q)\subset C_0([0,\infty))\rtimes \Z$, this exact sequence  becomes 
\begin{equation}\label{Cext}
\xymatrix{
 0\;\ar[r]&\; C_0((0,\infty))\rtimes \Z\;\ar@{^{(}->}[r]^{\quad\iota} &\; C_0(\C_q) 
\;\ar[r]^{\ \ \ \ev_0} &\; \C\;\ar[r] &\,0\,.}
\end{equation}

Note finally that we can identify the unit in $C^*_0(z,z^*)\dotplus \C\eins$ 
with the constant function $1 \in C(X\cup \{\infty\})$, \,$1(x)=1$. Then 
\begin{align} \label{Czz1}
&C^*_0(z,z^*)\dotplus \C\eins \\ 
&=\|\hsp\cdot\hsp \|\text{-}\mathrm{cls}\hs \Big\{ \sum_{\text{{\rm finite}}} f_k \, U^k\in C( X\hsp\cup\hsp \{\infty\})\rtimes \Z \;:\;   
k\in\Z ,\  f_k(0)=f_k(\infty)=0\ \hs\text{if }\hs  k\neq 0 \Big\}, \nonumber
\end{align} 
and the natural projection $ C^*_0(z,z^*)\dotplus \C\eins \lra  (C^*_0(z,z^*)\dotplus \C\eins) / C^*_0(z,z^*) \cong \C$ 
can be written 
\[     \label{evty}
\ev_\infty:C^*_0(z,z^*)\dotplus \C\eins  \lra \C, \quad 
\ev_\infty \Big(\sum_{\text{{\rm finite}}} f_k \hs U^k\Big) \,=\,  f_0(\infty). 
\]
The formula \eqref{ev0} remains unchanged for the unitalization  $C^*_0(z,z^*)\dotplus \C$. 

\section{K-theory} 

\subsection{K-theory of the quantum complex plane}  \label{secKCq}

There are several ways to compute the K-theory of $C_0(\C_q)$. 
To keep the paper elementary, we will use the 
standard six-term exact sequence in K-theory for C*-algebra extensions. 
We could have used Exel's generalized 
Pimsner--Voiculescu six-term exact sequence for generalized 
crossed product algebras defined by partial automorphisms \cite{Ex} 
but the gain would be minor at the cost of introducing more 
terminology. 

The C*-algebra extension \eqref{Cext} yields the following six-term exact sequence of K-theory: 
\begin{equation}  \label{six-term}
\xymatrixcolsep{3pc}
\xymatrix{  
 \ K_0 (C_0((0,\infty))\rtimes \Z ) \ \ar[r]^-{\iota_\ast} & 
  \ K_0 ( C_0(\C_q))\   \ar[r]^-{{\ev_0}_\ast} &
  \ K_0 (\C)\ \ar[d]^{\delta_{01}}\\
\ K_1 (\C)\ \ar[u]^{\delta_{10}} &
  \ K_1 ( C_0(\C_q))\  \ar[l]_-{\ \ {\ev_0}_\ast} &
\ K_1 (C_0((0,\infty))\rtimes \Z )\,.  \ar[l]_-{\ \ \iota_\ast}  
  }
\end{equation}

To resolve \eqref{six-term}, we need to know 
$K_0 (C_0((0,\infty))\rtimes \Z )$ and $K_1(C_0((0,\infty))\rtimes \Z )$. 
These K-groups can  easily obtained from the 
Pimsner--Voiculescu six-term exact sequence \cite[Theorem 10.2.1]{Bla} by 
noting that the automorphism $\alpha_q$ is homotopic 
to $\id=\alpha_1=\lim_{q\ra 1} \alpha_q$ so that the induced group homomorphisms yield 
$(\alpha_q)_* - \id=0$.  Another way of deducing these K-groups would be to 
consider a continuous field of C*-algebras (see e.g.\ \cite{Ro}), 
or to use the fact that $\Z$ acts freely and properly on $(0, \infty)$ so that 
$C_0((0,\infty))\rtimes \Z$ is Morita equivalent to 
$C((0,\infty)/\Z) \cong C(\mathrm{S}^1)$ (see e.g.\ \cite[Remark 4.16]{Will}). 
In any case, the outcome is 
\begin{equation} \label{K-cross0}
 K_0 (C_0((0,\infty))\rtimes \Z ) =\Z,  \qquad 
 K_1 (C_0((0,\infty))\rtimes \Z ) =\Z.
\end{equation}
The isomorphism  $K_1 (C_0((0,\infty))\rtimes \Z )\cong K_1 (C_0((0,\infty)))$ in 
the Pimsner--Voiculescu six-term exact sequence also shows that a  
generator of $K_1 (C_0((0,\infty))\rtimes \Z )$ 
may be  given by the $K_1$-class of the  unitary 
\[ \label{h}
 \e^{-2\pi\im h} \,\in\, C_0((0,\infty))\dotplus \C  \,\subset  \,
\big(C_0((0,\infty))\rtimes \Z\big) \dotplus \C,   
\]
where $h: [0,\infty) \ra \R$ denotes a continuous function such that $h(0)=1$ and $\underset{t\to\infty}{\lim} h(t)=0$. 
As $h\in C_0(\R_+)\subset C_0(\C_q)$ is a self-adjoint lift of the trivial projection $1\in\C$ under $\ev_0$, we get 
$\delta_{01}([1])= [\e^{-2\pi\im h}] \in K_1 (C_0((0,\infty))\rtimes \Z )$, see e.g.\ \cite[p.\ 172]{wegge}. 
Since $\delta_{01}$ maps a generator into a generator, it is an isomorphism and thus 
the adjacent homomorphisms ${\ev_0}_\ast$  and $\iota_\ast$ 
in \eqref{six-term} are 0. 

Inserting $K_0 (\C ) =\Z$, \,$K_1 (\C ) = 0$ and \eqref{K-cross0} into \eqref{six-term} 
yields $K_0 ( C_0(\C_q)\cong \Z$ and 
$K_1 ( C_0(\C_q)\cong 0$. 
Adjoining furthermore a unit to $C_0(\C_q)$, we have proven the following theorem: 
\begin{thm} \label{T1}
The $K$-groups of the C*-algebras $C_0(\C_q)$ and $C(\mathrm{S}^2_q)$ 
from Equations  \eqref{C0Cq} and \eqref{CSq}, respectively, 
are given by 
\begin{align}
 K_0(C_0(\C_q))&\cong \Z, & K_0(C(\mathrm{S}^2_q))&\cong \Z\oplus \Z, \\
K_1(C_0(\C_q))&\cong 0,
 &   K_1(C(\mathrm{S}^2_q))&\cong 0. 
\end{align}
\end{thm}

\begin{rem}  \label{R2} 
In Corallary \ref{C1} below, we will show that generators for $K_0(C_0(\C_q))$ are given 
by the $K_0$-classes $[P_{\pm 1}] -[1]$  
with the Bott projections  $P_{\pm 1}$  defined in \eqref{BP} and  \eqref{BPN}, 
and by the $K_0$-class $[R_1]$ 
with the Powers--Rieffel type projection $R_1$  defined in \eqref{PRP}. As a trivial consequence, 
each of these elements together with $[1]$ generate $K_0(C(\mathrm{S}^2_q))$. 
\end{rem}  

\begin{rem} 
Note that the $K$-groups  in Theorem \ref{T1} 
are isomorphic to the classical counterparts since  
$K_0(C_0(\C))\cong \Z$, \,$K_1(C_0(\C))= 0$, 
\,$K_0(C(\mathrm{S}^2))\cong \Z\oplus \Z$  and  $K_1(C(\mathrm{S}^2))=0$.  
\end{rem}  

\subsection{K-theory of C*-algebras generated by generic q-normal operators} 
\label{generic}

In this section, we describe the K-theory of the C*-algebra $C^*_0(z,z^*)$  defined in \eqref{Czz}, 
where $z$ is a $q$-normal operator such that $X:=\spec(|z|) \neq [0,\infty)$. By replacing $z$ by $t z$, \,$t>0$, 
we may assume that $1\notin\spec(|z|)$. From the $q$-invariance of $\spec(|z|)$ and 
the properties of the self-adjoint operator $A$ described below \eqref{zh}, 
we conclude that $Y:=\spec(|z|) \cap   (q,1) = \spec(A)$ is a compact subset of $(q,1)$ and 
$X = \{0\}\cup \underset{n\in\Z}{\cup} q^n \hs Y$, see \eqref{X}.  

Similar to the previous section, we consider the standard 
six-term exact sequence associated to the C*-algebra extension \eqref{ext}, i.e., 
\begin{equation}  \label{stX}
\xymatrixcolsep{3pc}
\xymatrix{  
 \ K_0 (C_0(X\!\setminus\!\{0\})\rtimes \Z ) \ \ar[r]^-{\iota_\ast} & 
  \ K_0 ( C^*_0(z,z^*))\   \ar[r]^-{{\ev_0}_\ast} &
  \ K_0 (\C)\ \ar[d]^{\delta_{01}}\\
\ K_1 (\C)\ \ar[u]^{\delta_{10}} &
  \ K_1 ( C^*_0(z,z^*))\  \ar[l]_-{\ \ {\ev_0}_\ast} &
\ K_1 (C_0(X\!\setminus\!\{0\})\rtimes \Z )\,.  \ar[l]_-{\ \ \iota_\ast}  
  }
\end{equation} 
Now we need to know the K-groups of the crossed product algebra $C_0(X\!\setminus\!\{0\})\rtimes \Z$. 
Clearly, the $\Z$-action on the disjoint union $X\hsp\setminus\hsp\{0\}=\underset{n\in\Z}{\cup} q^n \hs Y$ 
is free and proper, and the quotient space
$(X\hsp\setminus\hsp \{0\})/\Z$ can be identified with $Y$. It follows from \cite[Corollary 15]{G} that 
$C_0(X\!\setminus\!\{0\})\rtimes \Z \cong C(Y) \otimes \rmK(\ell_2(\Z) )$.  
By C*-stabilization, we have 
\[  \label{Kcross}
K_i(C_0(X\!\setminus\!\{0\})\rtimes \Z) \cong K_i(C(Y) \otimes \rmK(\ell_2(\Z) )) \cong K_i(C(Y)), \quad i=0,1, 
\] 
where a set of generators of $K_i(C_0(X\!\setminus\!\{0\})\rtimes \Z)$ 
is given by a set of generators of $K_i(C(Y))$ under the embeddings  
$C(Y)  \subset C_0(X\!\setminus\!\{0\}) \subset C_0(X\!\setminus\!\{0\})\rtimes \Z$. 
So we are reduced to determining the K-theory of $C(Y)$. 

The K-groups of $C(Y)$ for arbitrary compact planar sets $Y\subset \C$ are well known 
and a characterization of them can be found, e.g., in \cite[Section 7.5]{HR}. 
For an explicit description of the generators and later reference, we will introduce some notation 
in the next remark and then state the result in a proposition. 

\begin{rem}  \label{remJ}
Let $Y\subset (q,1)$ be a compact set and let 
$s\in (q,1)$ denote the maximum of $Y$. 
Consider the family  $\{ I_j: j\in J\}$  of connected components 
of $(q,s]\setminus Y$. These 
connected components are of course open intervals in $(q,1)$.  
If $Y$ has a finite number of connected components, say $n\in \N$, then $J$ has the same number of elements and 
we may choose $J=\{1,\ldots,n\}$. If $Y$ has an infinite (possibly uncountable) number of connected components,  
then $J$ is countable infinite and we may take $J=\N$. 
For each $j\in J$, we choose a $c_j\in I_j$ so that we arrive at the following situation: 
\[ \label{Ij}
Y^{\rmc}   :=   (q,1)\!\setminus \! Y   =   (s,1) \cup \underset{j\in J}{\cup}I_j, \ \  
 I_j  \cap   I_k   =  \emptyset \text{ if } j   \neq   k , \ \
c_j\in I_j   \subset    (q,s) . 
\]
Moreover,  we will frequently use the indicator function of a subset $A\subset \R$ given by 
$\chi_{A}(t) := 1$ for  $ t\in A$ and $0$ otherwise. 
Note that $\chi_{(x,y)}$ is a continuous projection for all $x,y\in Y^\rmc := (q,1)\setminus Y$,  i.e., 
$\chi_{(x,y)}\in C(Y)$ and  $(\chi_{(x,y)} )^2 =\chi_{(x,y)}  = (\chi_{(x,y)})^*$. 
\end{rem}

\begin{prop}  \label{KCY} 
Let $Y\subset (q,1)$ be a non-empty compact set. For all such sets,  
\[ \label{K1}
K_1(C(Y))\, = \, 0. 
\]
If $Y$ has $n\in \N$ connected components, then
\[ 
K_0(C(Y)) \, \cong \,  \Z^n.
\]
If $Y$ has an infinite number of connected components, then 
\[ \label{infinite}
K_0(C(Y))\,  \cong\,  \underset{n\in\N}{\oplus}\,  \Z \qquad \text{(infinite direct sum}). 
\] 
For any choice of real numbers $c_j\in I_j$ as in Remark \ref{remJ}, 
the equivalence classes of the projections $\chi_{(c_j,1)}\in C(Y)$, \,$j\in J$,  generate freely $K_0(C(Y))$. 
\end{prop} 
\begin{proof} 
Equation \eqref{K1} follows immediately from \cite[Propositions 7.5.2 and 7.5.3]{HR}  
since $\C\setminus Y$ has no bounded connected component.  

For a description of $K_0(C(Y))$ by a complete set of generators, 
we first identify projections $P\in \Mat_k(C(Y))$, \,$k\in\N$, 
with projection-valued continuous functions $P: Y \ra \Mat_k(\C)$. 
As in  \cite[Definition 7.5.1]{HR}, 
let $\check H^0(Y,\Z)$ denote the group of continuous, integer-valued functions on $Y$. 
By  \cite[Proposition 7.5.2]{HR}, the map 
\[  
h_0 : K_0(C(Y)) \lra \check H^0(Y,\Z), \qquad h_0([P])(t):=\mathrm{rank}(P(t)), 
\] 
is an group isomorphism. 
Since the continuous function $h_0([P])$ takes values in a discrete set, it is locally constant. 
By the compactness of $Y$, 
it can only have a finite number of jumps 
and each jump can only occur if the distance of 
neighboring points is greater than 0. 
Hence there exist  $c_{j_1}, \ldots, c_{j_{k_0}} \in (q,1)\setminus Y$ 
such that $c_{j_k}$ belongs to the connected component $I_{j_k}$ 
as described in \eqref{Ij}, $c_{j_1}<  \cdots < c_{j_{k_0}}$ 
and $h_0([P])$ is constant on $(c_{j_k}, c_{j_{k+1}})\cap Y$ for $k=1,\ldots, k_0$, where 
we set $c_{j_{k_0+1}}:=1$. 
Let $n_k\in\N_0$ such that  $h_0([P])(t) = n_k\in\N_0$ for  all $t\in (c_{j_k}, c_{j_{k+1}})\cap Y$.   
Then 
\[ \label{Psur}
h_0([P])= \msum{k=1}{k_0}  n_k \hs \chi_{(c_{j_k},c_{j_{k+1}})} = 
 n_{1}\hs  \chi_{(c_{j_1},1)} +  \msum{k=2}{k_0} \hs (n_{k} - n_{k-1})\hs \chi_{(c_{j_k},1)}. 
\]
Define $[p]\in K_0(C(Y))$ by 
\[  \label{psur} 
[p] :=  \msum{k=1}{k_0}  n_k \hs [\chi_{(c_{j_k},c_{j_{k+1}})}] 
 =   n_{1}\hs  [\chi_{(c_{j_1},1)}]  +  \msum{k=2}{k_0} \hs (n_{k} - n_{k-1})\hs [\chi_{(c_{j_k},1)}].  
\]
Since obviously $h_0([p]) = h_0([P])$, it follows from the injectivity of $h_0$ that $[p] = [P]$ in $K_0(C(Y))$. 

Now consider the group homomorphism 
\[ \label{Phi} 
\Phi :  \underset{j\in J}{\oplus}\,  \Z \lra K_0(C(Y)), \qquad 
\Phi( (g_j)_{j\in J}):= \overset{N}{\underset{k=1}{\mbox{$\sum$}}}  \, g_{n_k} \hs [\chi_{(c_{n_k},1)}],
\] 
where $N\in\N$ and $g_{n_1}, \ldots, g_{n_N}$ are 
the non-zero elements of $(g_j)_{j\in J}\in \oplus_{j\in J} \Z\,{\setminus}\hs \{0\}$. 
From the representation  \eqref{psur} of any $K_0$-element, 
one sees  immediately that $\Phi$ is surjective. The injectivity follows from the 
linear independence of the set  of functions $\{ \chi_{(c_{n_j},1)}: j\in J\}$ 
since $h_0\big(\Phi( (g_j)_{j\in J})\big)= \overset{N}{\underset{k=1}{\mbox{$\sum$}}}  \, g_{n_k} \hs \chi_{(c_{n_k},1)}$ 
and $\ker(h_0) =\{0\}$.  Hence $\Phi$ defines an isomorphism and 
$\{ [ \chi_{(c_{j},1)}  ] \hsp=\hsp \Phi((\delta_{ji})_{i\in J}) \hs:\hs j\hsp\in \hsp J\}$ yields a set of generators. 
\end{proof} 

Returning to the computation of the $K$-groups of $C^*_0(z,z^*)$, we can now insert the results 
of Proposition \ref{KCY} and Equation \eqref{Kcross} into the six-term exact sequence \eqref{stX}. 
Then the second line has only trivial groups in the corners, thus $K_1(C^*_0(z,z^*))\cong 0$. 
The remaining short exact sequence is obviously split exact with a splitting homomorphism 
\mbox{$\sigma :  K_0 (\C) \ra K_0 ( C^*_0(z,z^*))$} given by sending the generator $[1]\in K_0 (\C)$ to 
class of the continuous projection $\chi_{[0,q)} \in C_0(X)\subset C^*_0(z,z^*)$.  
Thus, to obtain $K_0(C^*_0(z,z^*))$, one only has 
to add one free generator to $K_0(C_0(X\!\setminus\!\{0\})\rtimes \Z) \cong K_0(C(Y))$, 
for instance, we may take $[\chi_{[0,q)}]= \sigma([1])$. This proves the following theorem. 

\begin{thm} \label{KCzz}
Let $z$ be a $q$-normal operator such that $X:=\spec(|z|) \neq [0,\infty)$ 
and assume without loss of generality that $q, 1\notin X$. 
If $(q,1) \cap X$ has $n\in \N$ connected components, then 
$$
K_0( C^*_0(z,z^*))\, \cong\, \Z^{n+1},\quad \quad 
K_1( C^*_0(z,z^*)) \,=\,0. 
$$
If $(q,1) \cap X$ has an infinite number of  connected components, then 
$$
K_0( C^*_0(z,z^*))\, \cong\, \underset{n\in\N}{\oplus}\,  \Z \quad \text{(infinite direct sum}),  \quad 
K_1( C^*_0(z,z^*)) \,=\,0. 
$$
A set of generators for $K_0( C^*_0(z,z^*))$ is given by $[\chi_{[0,q)}]$ and $[\chi_{(c_j,1)}]$, $j\in J$, 
where $J$ and  $c_j\in (q,1)$ are defined as in Remark \ref{remJ} for $Y:= (q,1)\cap X$.
\end{thm} 
\begin{rem}
For all $q$-normal operators $z$ such that $X:=\spec(|z|) \neq [0,\infty)$, we see by Theorem \ref{KCzz} that 
$K_0( C^*_0(z,z^*))$ contains more copies of $\Z$ than $K_0(C_0(\C))\cong \Z$ since $(q,1) \cap X$ 
has at least one connected component. 
This observation justifies the definition of $C_0(\C_q)$ as the C*-algebra generated by a 
$q$-normal operator $z$ such that $\spec(|z|) = [0,\infty)$ 
because only then the equalities $K_0(C_0(\C_q))= K_0(C_0(\C))$ and  $K_1(C_0(\C_q))= K_1(C_0(\C))$ hold. 

However, if one wants to consider the C*-algebras $C^*_0(z,z^*)$ from Theorem \ref{KCzz} also as algebras of 
continuous functions vanishing at infinity on a quantum complex plane, 
then, by \cite[Corollary~2.4]{CW}, all the abelian groups $\Z^n$, $n\in \N$, as well as $\underset{n\in\N}{\oplus}\,\Z$ 
can occur as a $K_0$-group 
of a quantum complex plane.   
\end{rem}

\subsection{Bott projections and Powers--Rieffel type projections}  
\label{BPRproj}

Our interest in the Bott projections lies in the observation that they can be viewed as representing 
noncommutative complex line bundles of any winding number. 
These projections are given by $2\times 2$-matrices whose entries are rational functions 
in the generators $z$ and $z^*$.  
Taking advantage of the noncommutativity of the involved crossed product algebras, 
we can also find non-trivial 1-dimensional projections belonging to the C*-algebra. 
The defining formulas for the 1-dimensional projections are 
completely analogous to the so-called Powers-Rieffel projections for the
irrational rotation C*-algebra (noncommutative torus) $C(\mathrm{S}^1)\rtimes \Z$. 
The Bott projections and Powers--Rieffel type projections exist for all 
C*-algebras $C^*_0(z,z^*)$ generated by a $q$-normal operator $z$, 
but for $C_0(\C_q)$ they take on an added significance because they are 
used to express all $K_0$-classes of  $C_0(\C_q)$. This will be shown in the next section 
by computing the index pairing. As can be seen in Theorem \ref{KCzz}, there are much more elementary 
projections in $C^*_0(z,z^*)$ if $\spec(|z|)\neq [0,\infty)$.  The relations between the 
Bott projections, Powers--Rieffel type projections and the elementary projections will be revealed 
in the next section by computing the index pairing. 

By classical Bott projections, we mean the following projections 
representing line bundles of winding number $\pm n\in\Z$ over the classical 
2-sphere \cite[Section 2.6]{GFV}: 
\begin{align*}
   p_{n} &:=  \mbox{$\frac{1}{1+ {z}^n \bar{z}^{n}}$}  \left(
\begin{array}{cc}
 \bar{z}^n {z}^{n}   &  \bar{z}^{n}\\[8pt]
 {z}^n  & 1
\end{array}\hsp\right)
 =   \mbox{$\frac{1}{1+ {z}^n \bar{z}^{n}}$} \left(
\begin{array}{c}
 \bar{z}^n \\  1
\end{array}\right) \left(\begin{array}{cc}  {z}^{n}  & \!\! 1 \end{array}\right) \hsp , \\[2pt]
   p_{-n} &:= \mbox{$\frac{1}{1+\bar{z}^n {z}^{n}}$}  \left(
\begin{array}{cc}
 {z}^{n} \bar{z}^n   &  {z}^{n}\\[8pt]
 \bar{z}^n  &  1
\end{array}\hsp\right)
 =  \mbox{$\frac{1}{1+\bar{z}^n {z}^{n}}$}   \left(
\begin{array}{c}
{z}^n  \\  1
\end{array}\right) \left(\begin{array}{cc} \bar{z}^{n} &\!\! 1\end{array}\right) 
\hsp ,\qquad n \in \N_0. 
 \end{align*}
Setting $v_n:= \mbox{$\frac{1}{\sqrt{1+ {z}^n \bar{z}^{n}}}$}   \left(\begin{array}{cc}  {z}^{n} & 1\end{array}\right)$ and 
$v_{-n}:= \mbox{$\frac{1}{\sqrt{1+ \bar{z}^n {z}^{n}}}$}    \left(\begin{array}{cc}  \bar {z}^{n} & 1\end{array}\right)$ for $n\in\N_0$,  
we can write  $p_k = v_k^* v_k$ for all $k\in\Z$. From the simple observation that $v_k v^*_k = 1$, 
it follows that $v_k$ is a partial isometry and thus 
$p_k$ is a projection. 

We apply the same ideas to define Bott projections in the non-commutative case, 
the only difference being the replacement of the unbounded continuous function $z\in C(\C)$ 
by the unbounded $q$-normal operator $z: \dom(z)\subset \hH\ra \hH$. For $n\in\N_0$, let 
\[
V_n:= \mbox{$\frac{1}{\sqrt{1+ {z}^n {z}^{*n}}}$}   \left(\begin{array}{cc}  {z}^{n} & 1\end{array}\right), \quad 
V_{-n}:= \mbox{$\frac{1}{\sqrt{1+ {z}^{*n} {z}^{n}}}$}    \left(\begin{array}{cc} {z}^{*n} & 1 \end{array}\right),\quad 
P_{\pm n} := V_{\pm n}^* V_{\pm n}. 
\]
Clearly, $V_{\pm n} \hs V_{\pm n}^* =1$, hence  $P_{\pm n} $ is a self-adjoint projection. 
Writing $z$ in its polar decomposition $z=U\hs |z|$ and using 
$U \hs f(|z|) \hs U^* = f(q|z|) =\alpha_q(f)(|z|)$ 
for every Borel function $f$ on $[0,\infty)$ \cite[Proposition 1]{CSS}, one computes 
\[ \label{BP}
   P_{n} = \left(
\begin{array}{ll}
\mbox{$\frac{q^{-n(n-1)}|z|^{2n}}{1+q^{-n(n-1)}\,|z|^{2n}}$} 
&\mbox{$\frac{q^{-\frac{n}{2}(n-1)}|z|^{n}}{1+q^{-n(n-1)}\,|z|^{2n}}$} \, U^{*n} \\[8pt]
\mbox{$\frac{q^{\frac{n}{2}(n+1)}\hs  |z|^n }{1+q^{n(n+1)}|z|^{2n}}$} \, U^{n} 
& \mbox{$\frac{1}{1+q^{n(n+1)}|z|^{2n}}$} 
\end{array}\hsp\right)\hsp ,
 \]
 \[ \label{BPN}
   P_{-n} = \left(
\begin{array}{ll} \mbox{$\frac{q^{n(n+1)}|z|^{2n}}{1+q^{n(n+1)}|z|^{2n}}$} 
&\mbox{$\frac{q^{\frac{n}{2}(n+1)}\hs  |z|^n}{1+q^{n(n+1)}|z|^{2n}}$} \, U^{n} \\[8pt]
 \mbox{$\frac{q^{-\frac{n}{2}(n-1)}\hs  |z|^n}{1+q^{-n(n-1)}|z|^{2n}}$}\, U^{*n}
& \mbox{$\frac{1}{1+q^{-n(n-1)}|z|^{2n}}$} 
\end{array}\hsp \right)\hsp . 
 \] 
By identifying the rational functions in $|z|$ with continuous functions on $\spec(|z|)$, 
we can view $P_n$ and $P_{-n}$ as 
projections in $\Mat_2(C^*_0(z,z^*) \dotplus \C)$. 
Furthermore, they present $K_0$-classes 
$[P_n] - [1], [ P_{-n}] - [1] \in K_0(C^*_0(z,z^*))$ and 
$[P_n], \,[ P_{-n}] \in K_0(C^*_0(z,z^*) \dotplus \C)$. 

Now we will construct 1-dimensional projections in $C^*_0(z,z^*)$ 
similar to the Powers--Rieffel projections in the irrational rotation C*-algebra \cite{R}.  
To this end, choose a continuous function 
$$
\phi : [q,1] \lra \R \quad \text{such \ that}  \quad   0\leq \phi \leq 1,\quad    \phi(q)= 0,\quad \phi(1)= 1, 
$$
and define for $n\in \N$ 
\begin{align}\nonumber &
h(t)\hsp :=\hsp  \left\{ 
\begin{array}{cl}
\!\!\sqrt{\phi(t)(1\hsp-\hsp\phi(t))},   & t\in[q,1], \\ 
  0, &  t\notin [q,1], 
\end{array} \right. \ \ 
f_n(t) \hsp:=\hsp   \left\{ \begin{array}{cl}
\phi(t) , & t\in[q,1], \\ 
1, & t \in (1,q^{-n+1}), \\
\!\!1\hsp- \hsp\phi(q^n t), & t\in [q^{-n+1} , q^{-n}],\\
0, &   t \notin [q,q^{-n}]. 
\end{array} \right. 
& \\[-14pt]  & &  \label{hfn}
\end{align}  
With $U$ denoting the unitary element from $ C_0(X)\rtimes \Z$ implementing the $\Z$-action, let 
\[ \label{PRP} 
R_n:= U^n\hs h + f_n + h\hs  U^{*n},\qquad n \in\N. 
\] 
Since $h,\,f_n\in C_0(X)$ and $h(0) = 0$, we have $R_n\in C^*_0(z,z^*)$. 
Obviously, $R_n^*=R_n$. Furthermore, direct computations show that $R_n^2=R_n$. 
Hence $[R_n]\in K_0(C^*_0(z,z^*))$ defines a $K_0$-class, and one can write $-[R_n] = [1-R_n] - [1]$  
with the 1-di\-men\-sional projection $1-R_n \in C^*_0(z,z^*)\dotplus \C$.

\section{K-homology and index pairings} 

\subsection{Fredholm modules} 

Even and odd Fredholm modules for a C*-algebra  $\A$ 
define equivalence classes in the $K$-homo\-logy groups 
$K^0(\A)$ and $K^1(\A)$, respectively. 
As we are interested in the index pairing of Fredholm modules with K-theory, 
and as it has been shown that the $K_1$-groups of C*-algebras generated by a 
q-normal operator are trivial, we will only consider even Fredholm modules which pair with 
$K_0$-classes. 

For our purposes, it suffices to describe an even Fredholm module as 
a pair of bounded *-representations $\pi_-$ and $\pi_+$  of $\A$ on a Hilbert space, say $\hH$, such that 
$\pi_+(a) - \pi_-(a)\in  \rmK(\hH)$ for all $a\hsp \in\hsp \A$. Its class in $K^0(\A)$ will be denoted by $[(\pi_-,\pi_+)]$. 
Let $p\in \Mat_N(\A)$ be a self-adjoint projection. Then 
$\hs\pi_+(p) \hs :\hs \pi_-(p)\hH^N \lra \pi_+(p)\hH^N$
is a Fred\-holm operator and the index map 
\[ \label{inding} 
\ip{[(\pi_-,\pi_+)]}{[p]} := \ind(\pi_+(p)\!\!\upharpoonright_{\pi_-(p)\hH^N})  
\]
defines a pairing between $K^0(\A)$ and $K_0(\A)$.  
Provided that $\pi_-(p) - \pi_+(p)$ is of trace class, 
the index pairing can be computed by 
the trace formula 
\[    
\ip{[(\pi_-,\hsp\pi_+)]}{\hsp[p]} = \Tr_{\hH} \big(\Tr_{\Mat_N(\A)} ( \pi_-(p) -  \pi_+(p) ) \big),         \label{IP}
\]
see e.g.\ \cite{C} or \cite{GFV}.  

For the index pairing \eqref{inding} to be well-defined, 
one does not need to assume that $\pi_+$ and $\pi_-$ are unital representations. 
In particular, consider the *-ho\-mo\-mor\-phisms  $\ev_\infty,\,\ev_0: C^*_0(z,z^*)\dotplus \C\lra \C$, where 
$\ev_\infty$ was defined in \eqref{evty} and $\ev_0$ denotes the extension of 
the homomorphism \eqref{ev0} to  $C^*_0(z,z^*)\dotplus \C$. Setting $\hH_- = \hH_+:=\C$, 
the  pairs $(\ev_\infty, 0)$ and $(\ev_0, 0)$ yield trivially Fredholm modules for $C^*_0(z,z^*)\dotplus \C$ 
and the trace formula \eqref{IP} reads 
\begin{align} \label{IP0} 
\ip{[(\ev_\infty,0)]}{\hsp[p]} = \Tr_{\Mat_N(\C)}\big(\ev_\infty(p)\big) , \qquad  
\ip{[(\ev_0,0)]}{\hsp[p]} = \Tr_{\Mat_N(\C)}\big(\ev_0(p)\big). 
\end{align}
Note that \eqref{IP0} computes the rank of the projections  $\ev_\infty(p), \,\ev_0(p)\in \Mat_N(\C)$. 
As $\C \cong (C^*_0(z,z^*)\dotplus \C) \hs / \hs C^*_0(z,z^*)$ corresponds to evaluating functions at the classical point $\infty$, 
we can view  the number $\Tr_{\Mat_N(\C)}\big(\ev_\infty(p)\big)$ as the rank of 
the noncommutative vector bundle determined by $p\in \Mat_N(C^*_0(z,z^*)\dotplus \C)$ in the spirit of the Serre--Swan theorem. 

Less trivially, the next proposition associates  a Fredholm module 
to any irreducible *-representation from \eqref{zh}.

\begin{prop} \label{P1}
For any $q$-normal operator $z$ and real number 
$y\in (q,1]\cap \spec(|z|)$, consider the Hilbert space representation $\pi_y : C^*_0(z,z^*)\dotplus \C \lra \rmB(\lZ)$
given on an orthonormal basis of $\lZ$ by 
\[  \label{repy}
\pi_y(U) \hs e_n = e_{n-1} ,\qquad \pi_y(f) e_n = f(q^n\hs y) \hs  e_n, \quad f\in C(\spec(|z|)\cup\{\infty\}). 
\]
Furthermore, let $\Pi_+$ and $\Pi_-$ denote the orthogonal projections onto 
the closed subspaces 
$\overline{\lin}\{e_k : k> 0\}$ and  \,$\overline{\lin}\{e_k : k\leq 0\}$ of $\lZ$, respectively, 
and define 
Hilbert space representations $\pi_0 ,\,\pi_\infty  : C^*_0(z,z^*)\dotplus \C \lra \rmB(\lZ)$ by 
\[ \label{repe}
\pi_0( a ) := \ev_0(a) \hs  \Pi_+ , \qquad  \pi_\infty( a ) := \ev_\infty(a) \hs  \Pi_-, \qquad a\in  C^*_0(z,z^*)\dotplus \C, 
\]
with the *-homomorphisms $\ev_0, \,\ev_\infty : C^*_0(z,z^*)\dotplus \C \lra \C$ defined 
at the end of Section~\ref{sec-qn}. 
Then the pair $(\pi_y, \pi_0\oplus\pi_\infty)$ is an even Fredholm module for 
$C^*_0(z,z^*)\dotplus \C$. 
\end{prop} 
\begin{proof}
Direct computation show that \eqref{repy} and \eqref{repe} define unital *-representations. By density and continuity, 
it suffices to show that $\pi_y(fU^n )- \pi_0(fU^n)  - \pi_\infty (fU^n)$ is compact for the 
generators $fU^n\in C^*_0(z,z^*)\dotplus \C$, $n\in\Z$. First, let $n\neq 0$. Then $f(0)=f(\infty) =0$ by \eqref{Czz1} and therefore 
$\pi_0(fU^n)  + \pi_\infty (fU^n) =0$ by  \eqref{repe}. 
The operator $\pi_y(f)$ is diagonal on the basis elements $e_k$ with  eigenvalues $f(q^k y)$ converging 
to $f(0)= f(\infty)=0$ as $k \ra  \pm \infty$. Therefore $\pi_y(f) $ is a compact operator and so is 
 $\pi_y(fU^n )- \pi_0(fU^n)  - \pi_\infty (fU^n) = \pi_y(f)\hs \pi_y(U^n)$. 

Next, let $n=0$. Then $\big(\pi_y(f )- \pi_0(f)  - \pi_\infty (f) \big)e_k = ( f(q^k y) - f(0) )\hs  e_k$ if $k>0$ and 
 $\big(\pi_y(f )- \pi_0(f)  - \pi_\infty (f) \big)e_k = ( f(q^k y) - f(\infty) )\hs  e_k$ if $k\leq 0$. Again, $\{ e_k:k\in\Z\}$ 
is a complete set of eigenvectors and the sequence of eigenvalues converges to $0$ since 
 $\underset{{k\ra \infty}}{\lim}  ( f(q^k y)\hsp  - \hsp f(0) ) \hsp = \hsp f(0)\hsp -\hsp  f(0)\hsp  =\hsp  0$ 
 and $\underset{k\ra -\infty}{\lim}( f(q^k y)\hsp  -\hsp  f(\infty) ) \hsp =\hsp  f(\infty)\hsp -\hsp  f(\infty)\hsp  =\hsp  0$. 
 Therefore $\pi_y(f )- \pi_0(f)  - \pi_\infty (f) $ yields again a compact operator. 
\end{proof} 

Recall that the equivalence relation in K-homology is defined by operator homotopy. 
If $y_1, \,y_2 \in  (q,1]\cap \spec(|z|)$ belong to the same connected component of  $\spec(|z|)\subset \R$, 
then $[0,1]\ni t\ra (\pi_{y_1 + t(y_2-y_1)} , \pi_0\oplus\pi_\infty)$ yields a operator homotopy between 
 $(\pi_{y_1}, \pi_0\oplus\pi_\infty)$ and  $(\pi_{y_2}, \pi_0\oplus\pi_\infty)$, hence these Fredholm modules 
 define the same class in $K^0(C^*_0(z,z^*)\dotplus \C)$. 
The non-degenerateness of the index pairing in Theorem \ref{Teo} will show that, if 
$y_1, \,y_2 \in  (q,1]\cap \spec(|z|)$ belong to different connected components, then the corresponding 
Fredholm modules  yield different K-homology classes.

\subsection{Index pairings for the quantum complex plane}  

The aim of this section is to compute the index pairing for the unital C*-algebra 
$C(\mathrm{S}^2_q)= C_0(\C_q)\dotplus \C$ viewed as the algebra of continuous functions on 
the one-point  compactification of the quantum complex plane. A family of projections describing $K_0$-classes 
was given in Section \ref{BPRproj}. However, since the computations apply to any $q$-normal operator, 
we state the result for a general C*-algebra $C^*_0(z,z^*)\dotplus \C$. 
Since, by Theorem \ref{T1},  $K_0(C(\mathrm{S}^2_q))\cong \Z\oplus\Z$ is torsion free, 
it follows from the universal coefficient theorem by Rosenberg and Schochet \cite{rs87}  that 
$K^0(C(\mathrm{S}^2_q))\cong \Z\oplus\Z$. Therefore it suffices to compute the index pairing 
for two generators of $K^0(C(\mathrm{S}^2_q))$. It turns out that one is given by Proposition \ref{P1}, 
and the other may be taken as $[(\ev_\infty,0)]$ from \eqref{IP0}. 
\begin{thm} \label{TIP} 
Let $z$ be a $q$-normal operator and $n\in \N$.  For $y\in (q,1]\cap \spec(|z|)$, 
consider the K-homology class 
$[(\pi_y, \pi_0\oplus\pi_\infty)]$ from Proposition \ref{P1}, 
and let $[(\ev_\infty,0)]$ denote the K-homology class  from Equation \eqref{IP0}. 
Then the index pairing between these 
K-ho\-mo\-logy classes 
and the K-theory classes of \,$C^*_0(z,z^*)\dotplus\hs \C$  
defined by the Bott projections $P_{\pm n}$ from Equations \eqref{BP} and \eqref{BPN}, 
and the Powers--Rieffel projections $R_n$ from Equation \eqref{PRP},  
is given by 
\begin{align*}
&\ip{[(\ev_\infty,0)]}{\hsp[P_{\pm n}]} = 1, & & \ip{[(\pi_y, \pi_0\oplus\pi_\infty)]}{\hsp[P_{\pm n}]} = \pm n, \\
 &\ip{[(\ev_\infty,0)]}{\hsp[R_n]} = 0,
& & \ip{[(\pi_y, \pi_0\oplus\pi_\infty)]}{\hsp[R_n]} = n.  
\end{align*}
Moreover, $\ip{[(\ev_\infty,0)]}{\hsp[1]} = 1$ and  $\ip{[(\pi_y, \pi_0\oplus\pi_\infty)]}{\hsp[1]} = 0$. 
\end{thm}
\begin{proof}
By setting $|z|=0$ 
and taking the limit $|z|\ra \infty$,  
one sees that the application of the evaluation maps \eqref{ev0} and \eqref{evty}  to the projections 
from  \eqref{BP}, \eqref{BPN} and \eqref{PRP} 
yields 
\begin{align}     \label{evP} 
&\ev_0(P_{\pm n}) = \begin{pmatrix}  0 & 0\\ 0 & 1  \end{pmatrix}, & &
\ev_\infty(P_{\pm n}) = \begin{pmatrix}  1& 0\\ 0 & 0  \end{pmatrix},\\
&\ev_0(R_n) = f_n(0)= 0, & & 
\ev_\infty(R_n) = f_n(\infty)= 0,   \label{evR}
\end{align} 
and also $\ev_0(1)=1=\ev_\infty(1)$.   In particular, by \eqref{IP0}, 
$$
\ip{[(\ev_\infty,0)]}{\hsp[P_{\pm n}]} = 1 , \qquad \ip{[(\ev_\infty,0)]}{\hsp[R_n]} = 0,\qquad  \ip{[(\ev_\infty,0)]}{\hsp[1]} = 1. 
$$
Furthermore, $\ip{[(\pi_y, \pi_0\oplus\pi_\infty)]}{\hsp[1]} =0$ 
by \eqref{IP} since $\pi_y(1) - (\pi_0\oplus\pi_\infty)(1)= 1-1=0$. 

We continue by computing the pairing $ \ip{[(\pi_y, \pi_0\oplus\pi_\infty)]}{\hsp[R_n]} $. 
From \eqref{repe} and \eqref{evR}, it follows that $(\pi_0\oplus\pi_\infty)(R_n)=0$.  
Therefore \eqref{IP} reduces to 
\[     	\label{traceR} 
\ip{[(\pi_y, \pi_0\oplus\pi_\infty)]}{\hsp[R_n]} = \Tr_{\lZ}  \big( \pi_y(U^n\hs h + f_n + h\hs  U^{*n} )\big) . 
\] 
As $\pi_y(U^n\hs h )e_k =  h(q^ky)e_{k-n}$ acts as a weighted shift operator, the trace of $\pi_y(U^n\hs h )$ vanishes, 
and so does the trace of its adjoint $\pi_y(h\hs U^{*n} ) = \pi_y((U^n h)^{*} )$. 
Therefore, computing the trace in \eqref{traceR} reduces to summing the matrix elements $\ip{e_k}{\pi_y(f_n)\hs e_k}$, 
$k\in\Z$. 
Applying \eqref{hfn} and \eqref{repy}, we get 
$$
\ip{[(\pi_y, \pi_0\oplus\pi_\infty)]}{\hsp[R_n]} = \Tr_{\lZ}  \big( \pi_y( f_n )\big) 
=  \phi(y) +\Big(\,\overset{n-1}{\underset{k=1}{\mbox{$\sum$}}}\hs  1 \,\Big) + 1- \phi(y) =n. 
$$

It remains to compute $ \ip{[(\pi_y, \pi_0\oplus\pi_\infty)]}{\hsp[P_{\pm n}]} $. 
From  \eqref{repe} and \eqref{evP}, it follows that 
$\Tr_{\Mat_2(C^*_0(z,z^*)\dotplus \C)}\big((\pi_0\oplus\pi_\infty )(P_{\pm n}) \big) = \Pi_+ + \Pi_- =1$. 
Thus,  for $[P_n]$ from \eqref{BP}, 
the trace formula \eqref{IP}  and the Hilbert space representation \eqref{repy} give 
\begin{align}   \label{ipPn} 
&\ip{[(\pi_y, \pi_0\oplus\pi_\infty)]}{\hsp[P_{n}]} 
= \Tr_{\lZ}\big( \pi_y(\mbox{$\frac{q^{-n(n-1)}|z|^{2n}}{1+q^{-n(n-1)}\,|z|^{2n}}$}  
+ \mbox{$\frac{1}{1+q^{n(n+1)}|z|^{2n}}$})  -1  \big) \\
&\quad  = \msum{k\in\Z}{} \Big( \mbox{$\frac{q^{-n(n-1)}(q^k\hs y)^{2n}}{1+q^{-n(n-1)}\,(q^k\hs y)^{2n}}$}   
+ \mbox{$\frac{1}{1+q^{n(n+1)}(q^{k}\hs y)^{2n}}$}  -1 \Big) 
\nonumber \\
&\quad = \msum{k=0}{\infty} 
\Big(\mbox{$\frac{q^{-n^2+n+2nk}\hs y^{2n}}{1+q^{-n^2+n+2nk}\hs y^{2n}}$}  
+ \big( \mbox{$\frac{1}{1+q^{n^2+n+2nk}\hs y^{2n}}$} -1\big)\Big) \nonumber\\
&\quad  \hspace{30pt} + \msum{k=1}{\infty}  
\Big(\big(\mbox{$\frac{q^{-n^2+n-2nk}\hs y^{2n}}{1+q^{-n^2+n-2nk}\hs y^{2n}}$} -1\big) 
+ \mbox{$\frac{1}{1+q^{n^2+n-2nk}\hs y^{2n}}$}  \Big) 
\nonumber\\
&\quad = \msum{k=0}{\infty} 
\Big(\mbox{$\frac{q^{-n^2+n+2nk}\hs y^{2n}}{1+q^{-n^2+n+2nk}\hs y^{2n}}$}  
- \mbox{$\frac{q^{n^2+n+2nk}\hs y^{2n}}{1+q^{n^2+n+2nk}\hs y^{2n}}$} \Big) 
+ \msum{k=1}{\infty}  
\Big(\mbox{$\frac{-1}{1+q^{-n^2+n-2nk}\hs y^{2n}}$}  
+ \mbox{$\frac{1}{1+q^{n^2+n-2nk}\hs y^{2n}}$}  \Big).  \nonumber
\end{align} 
Observe that 
\[  \label{sum1}
\msum{k=0}{\infty} \mbox{$\frac{q^{-n^2+n+2nk}\hs y^{2n}}{1+q^{-n^2+n+2nk}\hs y^{2n}}$}  
= 
\msum{k=0}{n-1} \mbox{$\frac{q^{-n^2+n+2nk}\hs y^{2n}}{1+q^{-n^2+n+2nk}\hs y^{2n}}$}  
+\msum{k=0}{\infty} \mbox{$\frac{q^{n^2+n+2nk}\hs y^{2n}}{1+q^{n^2+n+2nk}\hs y^{2n}}$}\, ,
\] 
where the second sum was obtained by shifting the summation index from $k$ to $n+k$. 
Similarly, 
\[
 \msum{k=1}{\infty} \mbox{$\frac{1}{1+q^{n^2+n-2nk}\hs y^{2n}}$} 
   = \msum{k=0}{n-1} \mbox{$\frac{1}{1+q^{-n^2+n+2nk}\hs y^{2n}}$}  +
   \msum{k=1}{\infty} \mbox{$\frac{1}{1+q^{-n^2+n-2nk}\hs y^{2n}}$}      \label{sum2}
\] 
by shifting the summation index from $k$ to $n-k$. 
Inserting \eqref{sum1} and \eqref{sum2} into \eqref{ipPn} yields 
\[ \label{n} 
\ip{[(\pi_y, \pi_0\oplus\pi_\infty)]}{\hsp[P_{n}]} 
= \msum{k=0}{n-1}\Big( \mbox{$\frac{q^{-n^2+n+2nk}\hs y^{2n}}{1+q^{-n^2+n+2nk}\hs y^{2n}}$}  
 + \mbox{$\frac{1}{1+q^{-n^2+n+2nk}\hs y^{2n}}$}  \Big) =  \msum{k=0}{n-1} \, 1 = n. 
\] 
Analogously, 
\begin{align*} 
&\ip{[(\pi_y, \pi_0\oplus\pi_\infty)]}{\hsp[P_{-n}]} 
= \Tr_{\lZ}\big( \pi_y(\mbox{$\frac{q^{n(n+1)}|z|^{2n}}{1+q^{n(n+1)}\,|z|^{2n}}$}  + \mbox{$\frac{1}{1+q^{-n(n-1)}|z|^{2n}}$})  -1  \big) \\ 
& = \msum{k\in\Z}{} \Big( \mbox{$\frac{q^{n(n+1)}(q^k\hs y)^{2n}}{1+q^{n(n+1)}\,(q^k\hs y)^{2n}}$}   
+ \mbox{$\frac{1}{1+q^{-n(n-1)}(q^{k}\hs y)^{2n}}$}  -1 \Big) \\
&= \msum{k=0}{\infty} 
\Big(
\mbox{$\frac{q^{n^2+n+2nk}\hs y^{2n}}{1+q^{n^2+n+2nk}\hs y^{2n}}$}
- \mbox{$\frac{q^{-n^2+n+2nk}\hs y^{2n}}{1+q^{-n^2+n+2nk}\hs y^{2n}}$}  
 \Big) 
+ \msum{k=1}{\infty}  
\Big(
 \mbox{$\frac{- 1}{1+q^{n^2+n-2nk}\hs y^{2n}}$}
 + \mbox{$\frac{1}{1+q^{-n^2+n-2nk}\hs y^{2n}}$}  
 \Big) \\
 & = - \ip{[(\pi_y, \pi_0\oplus\pi_\infty)]}{\hsp[P_{n}]}  = -n, 
\end{align*}
where we used \eqref{ipPn} and \eqref{n} in the last line. 
\end{proof} 

As announced in Remark \ref{R2}, we will now give explicit generators for $K_0(C_0(\C_q))$ 
and $K_0(C(\mathrm{S}^2_q))$. 

\begin{cor} \label{C1} 
Each of the $K_0$-classes $[R_1]$, 
$[P_{1}] - [1]$ and  $[P_{-1}] - [1]$ generates $K_0(C_0(\C_q))\cong \Z$,   
and a pair of generators for  $K_0(C(\mathrm{S}^2_q))\cong \Z\oplus \Z$ is obtained by adding the trivial 
class $[1]\in K_0(C(\mathrm{S}^2_q))$. 
\end{cor}

\begin{proof}
First we consider $K_0(C_0(\C_q))\cong \Z$. 
Since any multiple of a generator would yield a multiple of $1$ in the index pairing with K-homology classes, 
it is suffices to find an element in $K_0(C_0(\C_q))$ such that the pairing with a K-homology class yields~$\pm 1$. 
By Theorem \ref{TIP},  
$\ip{[(\pi_y, \pi_0\oplus\pi_\infty)]}{\hsp[R_1]} = 1$ and $\ip{[(\pi_y, \pi_0\oplus\pi_\infty)]\,}{[P_{\pm 1}]-[1]} = \pm 1$, 
therefore each of the three $K_0$-classes $[R_1]$, $[P_{1}] - [1]$ and $[P_{-1}] - [1]$ 
freely generates $K_0(C_0(\C_q))$. 
For a set of generators of $K_0(C(\mathrm{S}^2_q))$, one only needs to add the 
trivial class $[1]$. 
\end{proof} 

By analogy to the classical Bott projections, we may view the projective modules 
$C(\mathrm{S}^2_q)^2 P_n$ as the continuous sections of 
a non-commutative complex  line bundle over the quantum sphere $\mathrm{S}^2_q$ 
with winding number $n\in\Z$. Then $\ip{[(\pi_y, \pi_0\oplus\pi_\infty)]}{\hsp[P_{n}]} $ 
computes the winding number and the pairing with $[(\ev_\infty,0)]$ 
detects the rank of a noncommutative vector bundle (in the classical point $\infty$). 
Moreover, an isomorphism $\Z\cong K_0(C_0(\C_q))$ is given by 
$n \mapsto [P_n] - [1]$. 

As an application of the index pairing in Theorem \ref{TIP}, we will give an alternative description 
of non-commutative complex  line bundles by the 1-dimensional projections $R_n$, $n\in\N$, 
without the need of specifying equivalence relations in $K_0(C(\mathrm{S}^2_q))$. 

\begin{cor} \label{PnRn}
Let $n\in\N$. Given the Bott projections $P_{\pm n}$ from Equations \eqref{BP} and \eqref{BPN}, 
and the Powers--Rieffel projections $R_n$ from Equation \eqref{PRP}, 
the following equalities hold in $K_0(C(\mathrm{S}^2_q))$: 
$$
[P_n] = [1]+[R_n]= \left[ \begin{pmatrix}  R_n & 0\\ 0 & 1  \end{pmatrix}\right] ,\qquad [P_{-n}] = [1-R_n]. 
$$
\end{cor}
\begin{proof}
Since the $K$-homology classes $[(\ev_\infty,0)]$ and $[(\pi_y, \pi_0\oplus\pi_\infty)]$ from Theorem \ref{TIP} 
separate the generators of $K_0(C(\mathrm{S}^2_q))$ from Corollary \ref{C1}, it suffices to show 
that the index pairings coincide, which is straightforward. 
\end{proof}  
Clearly, there are no 1-dimensional projections in $C(\mathrm{S}^2)$ since 
$\mathrm{S}^2$ is connected, so the existence of the 1-dimensional projections 
$R_n$ can be regarded as a quantum effect. 
Note moreover that the index pairing with the $K_0$-classes $[R_n]$ 
reduces to the computation of very simple traces and is also much simpler than the computation of 
the index pairing with the $K_0$-classes determined by the Bott projections. 
In a certain sense,  one can say that the quantization of $\mathrm{S}^2$ leads to a 
significant simplification of the index pairing.

\subsection{Index pairings, generic case} 

In this section, we compute the index pairings for the C*-algebra $C^*_0(z,z^*)$  generated by a 
$q$-normal operator such that $X:=\spec(|z|) \neq [0,\infty)$. As in Section \ref{generic}, 
we assume that $1\notin\spec(|z|)$, and as in the previous section, we state the results 
for unitalization $C^*_0(z,z^*)\dotplus \C$ since the same results apply to the non-unital case 
after some minor modifications. The main difference to the previous section is that now the $K_0$-group 
is generated by simple projections of the type $\chi_A(|z|)\in C^*_0(z,z^*)\dotplus \C$, see 
Theorem \ref{KCzz}, where one has to add the trivial generator $[1]\in K_0(C^*_0(z,z^*)\dotplus \C)$. 

As for any compact set of real numbers, the connected components of $Y$ are closed intervals 
$K_\gamma := [a_\gamma, b_\gamma]$. 
As customary, 
we identify a singleton $\{y\}$ with the closed interval $[y,y]$. Let $\{ K_\gamma : \gamma \in\Gamma\}$ denote the set of the 
connected components of $Y$. For each $\gamma \in\Gamma$, choose a $y_\gamma\in K_\gamma$ 
and consider the Fredholm module 
\[  \label{Fg}
F_\gamma := (\pi_{y_\gamma}, \pi_0\oplus\pi_\infty)
\] 
from Proposition \ref{P1}.  
The next theorem shows that the index pairings with these 
K-homology classes together with the classes $ [(\ev_0,0)]$ and $[(\ev_\infty,0)]$ from \eqref{IP0} 
determine uniquely any  $K_0$-class 
of $C^*_0(z,z^*)\dotplus \C$. 

\begin{thm}  \label{Teo} 
Let $\{ K_\gamma : \gamma \in\Gamma\}$ and $F_\gamma$ be defined as above. 
The index pairing \eqref{IP} defines a non-degenerate pairing between the direct sum of even K-homology classes 
$$
\K := \Z [(\ev_0,0)]\, \oplus  \Big( \underset{\gamma\in\Gamma}{\oplus} \Z [F_\gamma] \Big) \oplus\, \Z[(\ev_\infty,0)]  
$$
and $K_0(C^*_0(z,z^*)\dotplus \C)$. The index pairing is determined by 
\begin{align}
&\ip{[(\ev_\infty,0)]}{[1]} =1, \quad \ip{[(\ev_\infty,0)]}{[\chi_{[0,q)}]} =0, \quad  \ip{[(\ev_\infty,0)]}{[\chi_{(c_j,1)}]} =0, 
\label{ip1}\\
&\ip{[(\ev_0,0)]}{[1]}=1,\quad \hspace{5pt} \ip{[(\ev_0,0)]}{[\chi_{[0,q)}]}=1,\quad\hspace{4pt}  \ip{[(\ev_0,0)]}{[\chi_{(c_j,1)}]}=0,
\label{ip2}\\
&  \ip{[F_\gamma]}{[1]}=0, \quad\hspace{29pt} \ip{[F_\gamma]}{[\chi_{[0,q)}]}=0, 
\label{ip3}\\
& \ip{[F_\gamma]}{[\chi_{(c_j,1)}]}=1\ \ \text{if}\ \ y_\gamma \in (c_j,1),\quad 
 \ip{[F_\gamma]}{[\chi_{(c_j,1)}]}=0 \ \ \text{if}\ \ y_\gamma \notin (c_j,1). \label{ip4}
\end{align}
\end{thm} 

\begin{proof} We first compute the index pairings for the generators of $K_0(C^*_0(z,z^*)\dotplus \C)$. 
Equations \eqref{ip1} and \eqref{ip2} are a simple consequence of  \eqref{IP0} by evaluating the 
1-di\-men\-sional projections in $\infty$ and $0$, respectively. 
From \eqref{repe} and the just mentioned evaluation maps, 
it also follows that 
\[ \label{Rep1}
(\pi_0\oplus\pi_\infty)(1) = 1,\quad 
(\pi_0\oplus\pi_\infty)(\chi_{[0,q)}) = \Pi_+, \quad 
(\pi_0\oplus\pi_\infty)(\chi_{(c_j,1)}) = 0. 
\] 
Moreover, \eqref{repy} yields $\pi_{y_\gamma}(\chi_A) \hs e_n =e_n$ 
\,if\, $q^n y_\gamma\in A\subset [0,\infty)$ 
and $\pi_{y_\gamma}(\chi_A)\hs  e_n =0$ otherwise. 
Therefore $\pi_{y_\gamma}(1) = 1$ and 
\[  \label{Rep2}
 \pi_{y_\gamma}(\chi_{[0,q)}) = \Pi_+, \quad 
\pi_{y_\gamma}(\chi_{(c_j,1)}) =0 \ \,\text{if} \ \, y_\gamma \notin (c_j,1), \quad 
\pi_{y_\gamma}(\chi_{(c_j,1)}) =\Pi_{e_0} \ \,\text{if} \ \, y_\gamma \in (c_j,1),
\]
where $\Pi_{e_0}$ denotes the 1-dimensional orthogonal projection onto $\lin\{e_0\}$. 
Combining \eqref{Rep1} and \eqref{Rep2} with \eqref{IP} yields for $F_\gamma$ from \eqref{Fg} 
\begin{align} 
& \ip{[F_\gamma]}{[1]}=\Tr_{\lZ} (1-1)=0, \quad\ip{[F_\gamma]}{[\chi_{[0,q)}]}=\Tr_{\lZ} (\Pi_+ -\Pi_+)=0, \label{Fchi} \\
& \ip{[F_\gamma]}{[\chi_{(c_j,1)}]}=\Tr_{\lZ} (0-0)=0\ \ \text{if} \ \ y_\gamma \notin (c_j,1), \label{Fchi0}\\
 &\ip{[F_\gamma]}{[\chi_{(c_j,1)}]}=\Tr_{\lZ} (\Pi_{e_0})=1\ \ \text{if} \ \ y_\gamma \in (c_j,1). \label{Fchi1}
\end{align} 
This finishes the proof of \eqref{ip1}--\eqref{ip4}. 

To show the non-degeneracy of the index pairing, let 
$$    
p:= l [\chi_{[0,q)}] + \msum{k=1}{N} n_k [\chi_{(c_{j_k},1)}] + m [1] \in K_0(C^*_0(z,z^*)\dotplus \C),\quad N\in\N ,\ \, l, n_k,m\in\Z, 
$$
and suppose that $\ip{F}{p} =0$ for all $F\in\K$. 
Pairing first with $[(\ev_\infty,0)]$ and then with $[(\ev_0,0)]$ gives $m=0$ and $l=0$ 
by  \eqref{ip1} and \eqref{ip2}. Thus $p=  \msum{k=1}{N} n_k [\chi_{(c_{j_k},1)}]$.

Without loss of generality, 
we may assume that $c_{j_1} > \ldots >c_{j_N}$.  As $c_{j_N}$ and $c_{j_{N-1}}$ belong to different 
connected components of $(q,1)\setminus Y$, there exists a $\gamma_N\in\Gamma$ 
such that $c_{j_{N-1}} > y_{\gamma_N}>c_{j_N}$. 
Then  \eqref{Fchi0} and \eqref{Fchi1} yield   
$0= \ip{[F_{\gamma_N}]}{[p]} = n_N $. 
Continuing inductively,  
choosing in each step a $\gamma_k\in\Gamma$ such that $c_{j_{K-1}} > y_{\gamma_k}>c_{j_k}$, 
where we set $c_{j_{0}}:=1$,  it follows that $n_N = \cdots =n_1=0$, hence $p=0$. 

Finally, let 
$$
F:= l [(\ev_0,0)] + \msum{k=1}{N} n_k  [F_{\gamma_k}]  + m [(\ev_\infty,0)] \in \K,\quad N\in\N ,\ \, l, n_k,m\in\Z, 
$$ 
and suppose that $\ip{F}{p} =0$ for all $p\in K_0(C^*_0(z,z^*)\dotplus \C)$. Similarly to the above, we may assume that 
$y_{\gamma_1} > \ldots > y_{\gamma_N}$. As each $y_{\gamma_k}$ belongs to a different connected component of $Y$, 
there exist ${j_k}\in J$ such that $y_{\gamma_k} >c_{j_k}> y_{\gamma_{k+1}}$  for $k=1,\ldots, N$, where 
$y_{\gamma_{N+1}}:= q$. 
From \eqref{ip1}, \eqref{ip2} and \eqref{ip4}, we obtain 
$0= \ip{F}{[\chi_{(c_{j_1},1)}] } = n_1$. Continuing by induction on $k=2,\ldots,N$, and applying in each step the same argu\-ment, 
we conclude that $n_2=\cdots = n_N=0$. Thus $F= l [(\ev_0,0)] + m [(\ev_\infty,0)]$. 
Now  \eqref{ip1}  and \eqref{ip2} imply first $0=  \ip{F}{[\chi_{[0,q)}]} = l$ and then $0=  \ip{F}{[1]} = m \ip{ [(\ev_\infty,0)]}{[1]} = m$, 
therefore $F=0$. 
\end{proof} 

As an application of the index pairing in Theorem \ref{Teo}, we will use elementary projections 
$\chi_A\in C^*_0(z,z^*)\dotplus \C$ to 
give an alternative description of the $K_0$-classes of the non-commutative complex line bundles determined by the 
Bott projections and Powers--Rieffel type projections from Section \ref{BPRproj}.  

\begin{cor} \label{corPRchi}
For $n\in\N$, let $P_{\pm n}$ denote the  Bott projections  defined in \eqref{BP} and \eqref{BPN}, 
and let $R_n$ denote the Powers--Rieffel type projections from \eqref{PRP}. 
Then the following equalities hold in $K_0(C^*_0(z,z^*)\dotplus \C)$. 
\begin{align}
&[P_n] = [1] +  [R_n]   = [1] + n \hs  [\chi_{(q,1)}] = \left[ \begin{pmatrix}  1 & 0\\ 0 & \chi_{(q^n,1)}  \end{pmatrix}\right], \label{chip}\\ 
&[P_{-n}] = [1] - [R_n]  = [1] -  n\hs   [\chi_{(q,1)}] = [1 - \chi_{(q^n,1)}] .     \label{chipn}
\end{align}   
\end{cor} 
\begin{proof} 
As  $1\notin X$ and thus $\{q^k : k\in\Z\}\cap X=\emptyset$
by the $q$-invariance of $X=\spec(|z|)$, 
we know that $\chi_{(q^n,\hs q^k)}$ is a projection in $C_0(X)\subset C^*_0(z,z^*)\dotplus \C$  
for all  $n,k\in\Z$, $n\hsp>\hsp k$. 

We will prove \eqref{chip} and \eqref{chipn}  by showing that the index pairings with the K-ho\-mo\-logy classes 
$[(\ev_0,0)]$,  $[(\ev_\infty,0)]$ and  $[F_\gamma]$, $\gamma \in\Gamma$, coincide. Then, 
by the non-degeneracy statement of Theorem~\ref{Teo}, 
Equations \eqref{chip} and \eqref{chipn}  yield identities 
in K-theory. 

Applying first Theorem \ref{TIP}, 
and then Theorem \ref{Teo} with $c_j$ replaced by $q$, 
one readily sees that 
\begin{align}   \label{PR1}
&\ip{[(\ev_\infty,0)]\hs }{\hsp[P_{ \pm n}]} = \ip{[(\ev_\infty,0)]\hs }{[1] \pm  [R_n] }  
=\ip{[(\ev_\infty,0)]\hs }{[1] \pm n [\chi_{(q,1)}] }  =1, \\
&\ip{[(\ev_0,0)]\hs }{\hsp[P_{ \pm n}]} 
= \ip{[(\ev_0,0)]\hs }{[1] \pm  [R_n] } = \ip{[(\ev_0,0)]\hs }{[1] \pm n [\chi_{(q,1)}] }  =1,   \label{PR3}\\
&\ip{[F_\gamma] \hs }{\hsp[P_{\pm n}]} =\ip{[F_\gamma]\hs }{[1] \pm  [R_n]} 
= \ip{[F_\gamma] \hs }{\hsp[1] \pm n [\chi_{(q,1)}]} =\pm n.  \label{PR2}
\end{align}
As a consequence, $[P_{\pm n}] = [1] \pm  [R_n] = [1] \pm n \hs  [\chi_{(q,1)}] $. 
By the obvious relations in K-theory, it now suffices to verify that 
$ [\chi_{(q^n,1)}]=n \hs  [\chi_{(q,1)}]$. Clearly, 
\[  \label{evi0}
\ip{[(\ev_{\hsp p},0)]\hs }{[\chi_{(q^n,1)}] } \,= \,0 \,= \, \ip{[(\ev_{\hsp p},0)]\hs }{n \hs[\chi_{(q,1)}] } , \quad  p\in\{0,\infty\}, 
\]
since all evaluation maps give 0. 
This also shows that $(\pi_0\oplus \pi_\infty)(\chi_{(q^n,1)})=0$ by \eqref{repe}. 
Furthermore, Equation \eqref{repy} implies that  
$\pi_y( \chi_{(q^n,1)})$ is the orthogonal projection onto $\lin\{ e_0,\ldots,e_{n-1}\}$ 
for all $y\in Y\subset (q,1)$. Therefore the index pairing \eqref{IP} 
yields 
$$  
\ip{[F_\gamma] \hs }{\hsp [\chi_{(q^n,1)}]} =  \Tr_{\lZ} \big( \pi_{y_\gamma}( \chi_{(q^n,1)})  \big) =n 
= n  \hs \Tr_{\lZ} \big( \pi_{y_\gamma}( \chi_{(q,1)})  \big)
 = \ip{[F_\gamma] \hs }{n\hs [\chi_{(q,1)}]}, 
$$
which completes the proof. 
\end{proof} 

Note that the projections $ \chi_{(q^n,1)}$ are utterly elementary, i.e., continuous functions 
with values in $\{0,1\}$. In particular, the computation of  index pairing reduces to its simplest possible form, 
namely to the calculation of a trace of a finite-dimensional projection. 
Also, by unitary equivalence of $K_0$-classes, we obtain from \eqref{chip} and \eqref{chipn} 
the following isomorphisms of finitely generated 
projective modules: 
\begin{align*}
(C^*_0(z,z^*)\dotplus \C)^2 \hs P_{-n} &\ \cong\  (C^*_0(z,z^*)\dotplus \C) \chi_{(q^n,1)}, \\
(C^*_0(z,z^*)\dotplus \C)^2 \hs P_n     &\ \cong\  (C^*_0(z,z^*)\dotplus \C)\, \oplus\, (C^*_0(z,z^*)\dotplus \C) \chi_{(q^n,1)}, 
\end{align*} 
where the right hand sides are considerably more simple. 

The interest in the projections $P_{\pm n}$ arose from the observation that 
they can be regarded as deformations of the classical Bott projections representing complex line bundles of winding number $\pm n$ 
over the 2-sphere. Recall that we defined the C*-al\-ge\-bra 
of the quantum 2-sphere as $C(\mathrm{S}^2_q) := C^*_0(z,z^*)\dotplus \C$, where 
$\spec(|z|)=[0,\infty)$, because only in that case the deformation preserves the classical K-groups. However, if one wants to 
view any $q$-normal operator $z$ as a deformation of the complex plane, then the deformations 
satisfying $\spec(|z|)\neq [0,\infty)$ lead to a substantial simplification of the description of complex line bundles and 
the index computation. 

\section*{Acknowledgements} 
This work was partially supported by the Simons--Foundation grant 346300, 
the Polish Government grant 3542/H2020/2016/2,  the EU funded grant H2020-MSCA-RISE-2015-691246-QUANTUM DYNAMICS, and CIC-UMSNH.

\end{document}